\documentclass[11pt]{amsart} 
\usepackage{hyperref} 
\usepackage{amsmath,amssymb}
\usepackage{mathtools}
\usepackage{amscd}
\usepackage{mathrsfs}
\usepackage{amsthm}
\usepackage{setspace}
\usepackage[all]{xy}
\usepackage{comment}
\usepackage[pdftex]{graphicx}
\usepackage{tikz}
\usetikzlibrary{arrows}
\usetikzlibrary{shapes.geometric,fit}
\usepackage{tikz-cd}
\usepackage{fancyhdr}

\theoremstyle{definition}
\newtheorem{thm}{Theorem}[subsection]
\newtheorem*{thm*}{Theorem}
\newtheorem{defi}[thm]{Definition}
\newtheorem*{defi*}{Definition}
\newtheorem*{acknowledge}{Acknowledgement}
\newtheorem{cor}[thm]{Corollary}
\newtheorem*{cor*}{Corollary}
\newtheorem{prop}[thm]{Proposition}
\newtheorem*{prop*}{Proposition}
\newtheorem{lem}[thm]{Lemma}
\newtheorem*{lem*}{Lemma}
\newtheorem{ex}[thm]{Example}
\newtheorem*{ex*}{Example}

\newtheorem{rem}[thm]{Remark}
\newtheorem*{rem*}{Remark}

\newtheorem*{hw*}{Homework}

\newcommand{\C}{\mathbb{C}}

\newcommand{\R}{\mathbb{R}}
\newcommand{\fix}{\mathrm{fix}}
\newcommand{\Z}{\mathbb{Z}}

\newcommand{\T}{\mathbb{T}}

\DeclareMathOperator{\chara}{\Delta}

\DeclareMathOperator{\Iso}{\mathrm{Iso}}

\DeclareMathOperator{\Funct}{Funct}
\DeclareMathOperator{\Span}{\mathrm{span}}
\DeclareMathOperator{\ev}{ev}

\DeclareMathOperator{\ab}{\mathrm{ab}}

\DeclareMathOperator{\Aut}{\mathrm{Aut}}

\def\i<#1>{\langle #1 \rangle}
\def\l<#1>{\left\langle #1 \right\rangle}

\makeatletter
\renewenvironment{proof}[1][\proofname]{\par
  \normalfont
  \topsep6\p@\@plus6\p@ \trivlist
  \item[\hskip\labelsep{\bfseries #1}\@addpunct{.}]\ignorespaces
}{%
  \endtrivlist
}
\renewcommand{\proofname}{\sc{Proof}}
\makeatother
\makeatletter
\newcommand*{\defeq}{\mathrel{\rlap{%
                     \raisebox{0.3ex}{$\m@th\cdot$}}%
                     \raisebox{-0.3ex}{$\m@th\cdot$}}%
                     =}
\makeatother

\title[The abelianizations of groupoid C*-algebras]{Quotients of \'etale groupoids and the abelianizations of groupoid C*-algebras}
\author{Fuyuta Komura}
\address{Department of Mathematics, Faculty of Science and Technology, Keio University
	3–14–1 Hiyoshi, Kohoku-ku, Yokohama, 223–8522, Japan}
\email{fuyuta.k@keio.jp}
\begin{document}
\maketitle
\begin{abstract}
	In this paper,
	we introduce quotients of \'etale groupoids.
	Using the notion of quotients, we describe the abelianizations of groupoid C*-algebras.
	As another application,
	we obtain a simple proof that effectiveness of an \'etale groupoid is implied by the full uniqueness property of its groupoid C*-algebra. 

\end{abstract}
\section{Introduction}	
The study of C*-algebras associated to \'etale groupoids, groupoid C*-algebras, was initiated by Renault in \cite{renault1980groupoid}.
Since then, many researchers have studied the relationship between \'etale groupoids and groupoid C*-algebras.
In the previous studies,
there are many results for C*-algebras associated to Hausdorff \'etale groupoids.
For a non-Hausdorff \'etale groupoid,
its C*-algebra seems not to be studied sufficiently.
However, non-Hausdorff groupoids naturally arise as mentioned in \cite{Connes82asurvey}, \cite{ExelPardo2016} and so on.
Exel pointed out that some results known for Hausdorff \'etale groupoids do not necessarily hold for non-Hausdorff groupoids in \cite{ExelnonHausdorff}.
In \cite{LisaRuyAidan}, the authors treat simplicity of groupoid C*-algebras associated to non-Hausdorff \'etale groupoids.

In this paper, we calculate the abelianization of a groupoid C*-algebra.
For a discrete group $\Gamma$,
the abelianization $C^*(\Gamma)^{\ab}$ of its group C*-algebra $C^*(\Gamma)$ is isomorphic to $C^*(\Gamma^{\ab})$, where $\Gamma^{\ab}$ is the abelianization of $\Gamma$.
Furthermore,
$C^*(\Gamma^{\ab})$ is isomorphic to $C(\widehat{\Gamma^{\ab}})$,
where $\widehat{\Gamma^{\ab}}$ is the Pontryagin dual of $\Gamma^{\ab}$.
It is natural to consider an \'etale groupoid analogy.
For an \'etale groupoid $G$,
we construct an \'etale groupoid $G^{\ab}$ and a topological groupoid $\widehat{G^{\ab}}$ so that $C^*(G)^{\ab}\simeq C^*(G^{\ab})\simeq C_0(\widehat{G^{\ab}})$ holds.
In order to construct $G^{\ab}$,
we introduce the notion of quotient \'etale groupoids.
A quotient \'etale groupoid often becomes non-Hausdorff even if the original \'etale groupoid is Hausdorff.
Therefore,
we treat not necessarily Hausdorff \'etale groupoids and their C*-algebras,
which are defined by Connes \cite{Connes82asurvey}.
As a byproduct,
we obtain a simple proof that effectiveness of an \'etale groupoid is implied by the full uniqueness property of its groupoid C*-algebra (see Corollary \ref{uniqueness property implies effectiveness}).
We remark that this result has been shown in \cite{Brown2014} for Hausdorff \'etale groupoids in a different way and the proof in \cite{Brown2014} seems to work for non-Hausdorff \'etale groupoids.

This paper is organized as follows.
In Section 1,
we recall definitions and basic facts about not necessarily Hausdorff \'etale groupoids and their C*-algebras.

In Section 2,
we introduce the notion of quotient \'etale groupoids and show some applications.
Using quotients,
we obtain a simple proof that full uniqueness property of a groupoid C*-algebra induces effectiveness of an \'etale groupoid. 

In Section 3,
for an \'etale groupoid $G$,
we construct an \'etale abelian group bundle $G^{\ab}$ through quotients.
Finally, we show that the abelianization of a groupoid C*-algebra is isomorphic to the C*-algebra associated to $G^{\ab}$.
Since the abelianization of a C*-algebra is commutative,
it is isomorphic to $C_0(X)$ for some locally compact Hausdorff space $X$ by the Gelfand-Naimark theorem.
We show that the abelianization of a groupoid C*-algebra $C^*(G)$ is isomorphic to $C_0(\widehat{G^{\ab}})$,
where $\widehat{G^{\ab}}$ is introduced in this paper.

We obtain some results by using quotients of \'etale groupoids,
which are not necessarily Hausdorff.
We expect that this paper stimulates the study of non-Hausdorff \'etale groupoids.

\begin{acknowledge}
The author would like to thank his supervisor, Prof.\ Takeshi Katsura, for fruitful discussions and guidance.
\end{acknowledge}

\section{\'Etale groupoids and groupoid C*-algebras}

In this section,
we recall the notions of \'etale groupoids and groupoid C*-algebras.
We refer to \cite{renault1980groupoid}, \cite{paterson2012groupoids} and \cite{asims} for details.

\subsection{\'Etale groupoids}

A groupoid is a set $G$ together with a distinguished subset $G^{(0)}\subset G$,
source and range maps $s,r\colon G\to G^{(0)}$ and a multiplication 
\[
G^{(2)}\defeq \{(\alpha,\beta)\in G\times G\mid s(\alpha)=r(\beta)\}\ni (\alpha,\beta)\mapsto \alpha\beta \in G
\]
such that
\begin{enumerate}
	\item for all $x\in G^{(0)}$, $s(x)=x$ and $r(x)=x$ hold,
	\item for all $\alpha\in G$, $\alpha s(\alpha)=r(\alpha)\alpha=\alpha$ holds,
	\item for all $(\alpha,\beta)\in G^{(2)}$, $s(\alpha\beta)=s(\beta)$ and $r(\alpha\beta)=r(\alpha)$ hold,
	\item if $(\alpha,\beta),(\beta,\gamma)\in G^{(2)}$,
	we have $(\alpha\beta)\gamma=\alpha(\beta\gamma)$,
	\item\label{inverse} every $\gamma \in G$,
	there exists $\gamma'\in G$ which satisfies $(\gamma',\gamma), (\gamma,\gamma')\in G^{(2)}$ and $s(\gamma)=\gamma'\gamma$ and $r(\gamma)=\gamma\gamma'$.   
\end{enumerate}
Since the element $\gamma'$ in (\ref{inverse}) is uniquely determined by $\gamma$,
$\gamma'$ is called the inverse of $\gamma$ and denoted by $\gamma^{-1}$.
We call $G^{(0)}$ the unit space of $G$.
A subgroupoid of $G$ is a subset of $G$ which is closed under the inversion and multiplication. 
For $U\subset G^{(0)}$, we define $G_U\defeq s^{-1}(U)$ and $G^{U}\defeq r^{-1}(U)$.
We define also $G_x\defeq G_{\{x\}}$ and $G^x\defeq G^{\{x\}}$ for $x\in G^{(0)}$.
The isotropy bundle of $G$ is denoted by $\Iso(G)\defeq\{\gamma\in G\mid s(\gamma)=r(\gamma)\}$.
If $G$ satisfies $G=\Iso(G)$,
$G$ is called a group bundle over $G^{(0)}$.
A group bundle $G$ is said to be abelian if $G_x$ is an abelian group for all $x\in G^{(0)}$.

A topological groupoid is a groupoid equipped with a topology where the multiplication and the inverse are continuous.
Note that the source map and range map of a topological groupoid are continuous.

\begin{defi}\label{def of etale}
	A topological groupoid $G$ is said to be \'etale if
	\begin{enumerate}
		\item the unit space $G^{(0)}\subset G$ is a locally compact Hausdorff space with respect to the relative topology of $G$,
		\item the source map $s\colon G\to G^{(0)}$ is locally homeomorphic (i.e.\ for all $\alpha\in G$, there exists an open neighborhood $U\subset G$ of $\alpha$ such that $s(U)\subset G^{(0)}$ is open and $s|_U$ is a homeomorphism onto $s(U)$).
	\end{enumerate}
\end{defi}
An \'etale topological groupoid is called an \'etale groupoid in short.
In this paper,
we assume that the unit space of an \'etale groupoid is a locally compact Hausdorff space. 
We \textbf{do not} assume that an \'etale groupoid is a Hausdorff space as a topological space.

Note that a local homeomorphism is an open map.
If $s$ is locally homeomorphic,
then $r$ is also locally homeomorphic since $r(\gamma)=s(\gamma^{-1})$ holds for all $\gamma\in G$.
By Definition \ref{def of etale},
the family of all locally compact Hausdorff open subsets of $G$ is an open basis for the topology of $G$.
An \'etale groupoid $G$ is said to be effective if $G^{(0)}=\Iso(G)^{\circ}$, where $\Iso(G)^{\circ}$ denotes the interior of $\Iso(G)$. 

In some papers, the condition that the source map $s\colon G\to G^{(0)}$ is locally homeomorphic in Definition \ref{def of etale} is replaced by the condition that the source map $s\colon G\to G$ is locally homeomorphic.
As in Proposition \ref{local homeo}, these definitions are equivalent.

\begin{prop}\label{local homeo}
	Let $G$ be an \'etale groupoid.
	The unit space $G^{(0)}$ is an open subset of $G$.
	Furthermore,
	the source and range maps $s, r$ are local homeomorphisms as maps from $G$ to $G$.
\end{prop}

\begin{proof}
	First,
	we show that $G^{(0)}$ is an open subsets of $G$.
	Take $x\in G^{(0)}$ arbitrarily.
	Then,
	there exists an open subset $U\subset G$ with $x\in U$ such that $s(U)$ is an open subsets of $G^{(0)}$ and $s|_U$ is a homeomorphism onto $s(U)$.
	Define $V\defeq U\cap s^{-1}(G^{(0)}\cap U)$.
	Observe that $V$ is an open subset of $G$ with $x\in V$.
	Moreover,
	observe that $V\subset G^{(0)}$.
	Indeed, for every $\gamma\in V$,
	we have $s(\gamma)\in U$ and $s(\gamma)=s(s(\gamma))$.
	Since $s|_U$ is injective,
	it follows that $\gamma=s(\gamma)\in G^{(0)}$.
	Therefore,
	we have shown $V\subset G^{(0)}$ and this implies that $G^{(0)}$ is an open subset of $G$.
	Now, the second assertion can be easily checked.
	\qed
\end{proof}

\begin{defi}
	Let $G$ be an \'etale groupoid.
	A subset $U\subset G$ is called a bisection if both $s|_U$ and $r|_U$ are injective.
\end{defi}

	For an \'etale groupoid $G$,
	an open bisection of $G$ is a locally compact Hausdorff space because it is homeomorphic to open subset of $G^{(0)}$ and we assume that $G^{(0)}$ is locally compact Hausdorff.
	Note that the set of all open bisections composes an open basis of $G$.

\begin{ex}
	Every locally compact Hausdorff space is regarded as a Hausdorff \'etale groupoid whose unit space coincides with the whole space.  
\end{ex}

\begin{ex}
	Every topological group is regarded as a topological groupoid whose unit space is a singleton.
	A topological group is discrete if and only if it is \'etale as a topological groupoid.
\end{ex}

\begin{ex}\label{action}
	Let $X$ be a locally compact Hausdorff space and $\Gamma$ be a discrete group.
	We denote by $\Aut(X)$ the set of all homeomorphisms on $X$,
	which is a group under the composition.
	An action of $\Gamma$ on $X$ is a group homomorphism $\alpha\colon \Gamma\ni s\mapsto\alpha_s\in\Aut(X) $ and written $\alpha\colon \Gamma\curvearrowright X$.
	Now,
	we construct the groupoid associated to an action $\alpha\colon \Gamma\curvearrowright X$.
	Define $\Gamma \ltimes_{\alpha}X\defeq \Gamma\times X$ as a topological space.
	The unit space of $\Gamma \ltimes_{\alpha}X$ is $X$,
	which is identified with the subset of $\Gamma \ltimes_{\alpha}X$ via an inclusion $X\ni x\mapsto (e,x)\in \Gamma \ltimes_{\alpha}X$.
	The source map and range map are defined by $s((t,x))=x$ and $r((t,x))=\alpha_t(x)$ respectively for $(t,x)\in \Gamma \ltimes_{\alpha}X$.
	For a pair $(t_1, y),(t_2, x)\in \Gamma \ltimes_{\alpha}X$ with $y=\alpha_{t_2}(x)$,
	their multiplication is defined by $(t_1,y)\cdot(t_2,x)\defeq (t_1t_2,x)$.
	An inverse is given by $(t,x)^{-1}=(t^{-1},\alpha_t(x))$.
	Then,
	$\Gamma \ltimes_{\alpha}X$ is a Hausdorff \'etale groupoid.
\end{ex}

\begin{prop}[{\cite[Proposition 2.2.4]{paterson2012groupoids}}]\label{multi} 
	Let $G$ be an \'etale groupoid and $U,V\subset G$ be open sets.
	Then,
	a set $UV\defeq \{\alpha\beta\in G \mid\alpha\in U,\beta\in V, s(\alpha)=r(\beta)\}\subset G$ is an open set. 
	Furthermore,
	if $U,V\subset G$ are open bisections,
	$UV$ is also an open bisection.
\end{prop}

\begin{proof}
	Take $\gamma\in UV$ and an open bisection $W\subset G$ with $\gamma\in W$.
	Then,
	there exist $\alpha\in U$ and $\beta\in V$ such that $\gamma=\alpha\beta$.
	By the continuity of the multiplication of $G$,
	there exist open bisections $U_1,V_1\subset G$ with $\alpha\in U_1\subset U,\beta\in V_1\subset V$ and $U_1V_1\subset W$.
	Note that $r(U_1V_1)=r(U_1\cap s^{-1}(r(V_1)))\subset G^{(0)}$ is an open subset.
	Therefore,
	$U_1V_1=r^{-1}(r(U_1V_1))\cap W\subset G$ is open.
	Now,
	we have $\gamma\in U_1V_1\subset UV$,
	so $UV\subset G$ is an open subset.
	
	Assume that $U,V\subset G$ are open bisections.
	One can show that $s|_{UV}$ and $r|_{UV}$ are injective.
	Therefore,
	$UV$ is an open bisection.
	\qed 
\end{proof}

\begin{defi}
	Let $G$ be a groupoid.
	A subset $F\subset G^{(0)}$ is said to be invariant if $s(\gamma)\in F$ implies $r(\gamma)\in F$ for all $\gamma$.
	A point $x\in G^{(0)}$ is called a fixed point if $\{x\}\subset G^{(0)}$ is invariant. 
\end{defi}

Note that a set $F\subset G^{(0)}$ is invariant if and only if $G^{(0)}\setminus F$ is invariant.
If $F\subset G^{(0)}$ is invariant,
then $G_F=G_F\cap G^F\subset G$ is a subgroupoid whose unit space is $F$.

\begin{prop}\label{fixed points are closed}
	Let $G$ be an \'etale groupoid.
	Then,
	the set of all fixed points $F\subset G^{(0)}$ is a closed subset.
\end{prop}

\begin{proof}
	We show that $G^{(0)}\setminus F \subset G^{(0)}$ is an open set.
	Take $x\in G^{(0)}\setminus F$.
	Then,
	there exists $\gamma\in G$ such that $x=s(\gamma)$ and $x\not=r(\gamma)$.
	Take an open bisection $U$ which contains $\gamma$.
	Let $S_U\colon s(U)\to r(U)$ denote a homeomorphism defined by $S_U(s(\alpha))=r(\alpha)$ for each $\alpha\in U$.
	Since $G^{(0)}$ is Hausdorff,
	there exist open sets $U_1,V_1\subset G^{(0)}$ such that $s(\gamma)\in U_1$ , $r(\gamma)\in V_1$ and $U_1\cap V_1=\emptyset$.
	By the continuity of $S_U$,
	there exists an open set $U_2\subset U$ such that $\gamma\in U_2$ and $S_U(U_2)\subset V_1$.
	Now,
	one can see $U_2\subset G^{(0)}\setminus F$.
	Therefore,
	$G^{(0)}\setminus F\subset G^{(0)}$ is an open set.\qed
\end{proof}

We will use the next proposition for the set of all fixed points.

\begin{prop}\label{restriction to inv is etale}
	Let $G$ be an \'etale groupoid and $U,F\subset G^{(0)}$ be an invariant open and closed subset respectively.
	Then,
	$G_U\subset G$ is an open subgroupoid of $G$ and an \'etale groupoid in the relative topology.
	Similarly,
	$G_F\subset G$ is a closed subgroupoid of $G$ and an \'etale groupoid in the relative topology.
\end{prop}

\begin{proof}
	Observe that $U$ and $F$ are locally compact Hausdorff spaces in the relative topology of $G^{(0)}$.
	Now, it is clear that $G_U$ and $G_F$ are \'etale groupoids.
	\qed
\end{proof}
In particular,
if $x\in G^{(0)}$ is a fixed point,
then $G_x\subset G$ is a discrete group.

\subsection{\'Etale groupoid C*-algebras}
Following Connes's idea in \cite{Connes82asurvey},
we associate a C*-algebra to an \'etale groupoid which is not necessarily Hausdorff.

Let $G$ be an \'etale groupoid.
For an open Hausdorff subset $U\subset G$,
we denote the set of all continuous functions with compact support on $U$ by $C_c(U)$.
We regard an element in $C_c(U)$ as an element in $\Funct(G)$, the vector space of all complex valued functions on $G$, by defining it to be $0$ outside of $U$.
We define $\mathcal{C}(G)\defeq \Span \bigcup_U C_c(U)\subset \Funct(G)$, where the union is taken over all open Hausdorff subsets $U\subset G$.

	If $G$ is Hausdorff,
	then $\mathcal{C}(G)$ coincides with $C_c(G)$.
	If $G$ is not Hausdorff,
	an element in $\mathcal{C}(G)$ may not be continuous.

\begin{prop}\label{supported on bisection}
	Let $G$ be an \'etale groupoid.
	Take an open basis $\{U_i\}_{i\in I}$ of $G$ consisting of open Hausdorff subsets.
	Then,
	$\mathcal{C}(G)$ is the linear span of $\bigcup_{i\in I}C_c(U_i)$.
	In particular,
	$\mathcal{C}(G)$ is the linear span of $\bigcup_{U}C_c(U)$,
	where the union is taken over all open bisections of $G$.
\end{prop}

\begin{proof}
	This follows from a partition of unity argument.\qed
\end{proof}

\begin{defi}
	Let $G$ be an \'etale groupoid.
	Recall that $\mathcal{C}(G)$ is equipped with a structure of $\C$-vector space by pointwise addition and scalar multiplication.
	The multiplication $f*g\in \mathcal{C}(G)$ and involution $f^*\in \mathcal{C}(G)$ of $f,g\in \mathcal{C}(G)$ are defined by
	\[
	f*g(\gamma)=\sum_{\beta\in G_{s(\gamma)}}f(\gamma\beta^{-1})g(\beta),\,\,
	f^*(\gamma)=\overline{f(\gamma^{-1})}.
	\]
	Then, $\mathcal{C}(G)$ is a *-algebra under these operations.
\end{defi}
One can see that $C_c(G^{(0)})\subset \mathcal{C}(G)$ is a *-subalgebra.

\begin{lem}\label{universal norm}
	Let $G$ be an \'etale groupoid and $f\in \mathcal{C}(G)$.
	Then, there exists $C_f\geq 0$ such that $\lVert\rho (f)\rVert \leq C_f$ for all Hilbert spaces $H$ and *-homomorphisms $\rho\colon \mathcal{C}(G)\to B(H)$.
\end{lem}

\begin{proof}
	We may assume that $f\in C_c(U)$ for some open bisection $U\subset G$.
	One can see that $f^**f\in C_c(G^{(0)})$.
	Since $C_c(G^{(0)})$ is a union of commutative C*-algebras,
	we have $\lVert\rho(h)\rVert\leq \sup_{x\in G^{(0)}}\lvert h(x)\rvert$ for all $h\in C_c(G^{(0)})$. 
	Then,
	we obtain $\lVert\rho(f)\rVert^2=\lVert\rho( f^**f)\rVert\leq \sup_{x\in G^{(0)}} \lvert f^**f(x)\rvert<\infty$.
	\qed
\end{proof}

The universal norm of $f\in \mathcal{C}(G)$ is defined by
\[
\lVert f\rVert \defeq \sup\{\lVert \rho(f) \rVert \mid \text{$\rho\colon \mathcal{C}(G)\to B(H)$ is a *-representation}\}.
\]
By Lemma \ref{universal norm},
the universal norm takes values in $[0,\infty)$.
Since the left regular representations of $\mathcal{C}(G)$ induces a faithful *-representation of $\mathcal{C}(G)$,
the universal norm becomes a C*-norm (see \cite[Section 4]{LisaRuyAidan}).
The completion of $\mathcal{C}(G)$ by universal norm is denoted by $C^*(G)$.
We shall remark that every *-representation of $\mathcal{C}(G)$ induces the *-representation of $C^*(G)$. 
Note that the inclusion $C_c(G^{(0)})\subset \mathcal{C}(G)$ extends to $C_0(G^{(0)})\subset C^*(G)$.

\begin{prop}\label{canonical surjection}
	Let $G$ be an \'etale groupoid and $F\subset G^{(0)}$ be a closed invariant set.
	Then,
	the restriction $\mathcal{C}(G)\ni f\mapsto f|_{G_F}\in \mathcal{C}(G_F)$ extends to the surjective *-homomorphism $C^*(G)\to C^*(G_F)$.
\end{prop}

\begin{proof}
	First,
	we check that $f|_{G_F}\in \mathcal{C}(G_F)$ for all $f\in \mathcal{C}(G)$.
	We may assume that $f\in C_c(U)$ for some open Hausdorff subset $U\subset G$,
	since $\mathcal{C}(G)$ is spanned by $\bigcup_U C_c(U)$,
	where the union is taken over all open Hausdorff subsets $U\subset G$.
	Defining $V\defeq G_F\cap U$, $V$ is a Hausdorff open subset of $G_F$.
	Then $f|_{G_F}$ is contained in $C_c(V)\subset \mathcal{C}(G_F)$.
	
	Direct calculations imply that the restriction $\mathcal{C}(G)\ni f\mapsto f|_{G_F}\in \mathcal{C}(G_F)$ is a *-homomorphism.
	
	Next,
	we show that the restriction $\mathcal{C}(G)\ni f\mapsto f|_{G_F}\in \mathcal{C}(G_F)$ is surjective.
	Note that $\{G_F\cap U\mid \text{$U\subset G$ is an open Hausdorff subset}\}$ is an open basis of $G_F$.
	Take an open Hausdorff subset $U\subset G_F$ and $f\in C_c(G_F\cap U)$ arbitrarily.
	Put $V\defeq G_F\cap U$.
	Since $V\subset U$ is a closed subset of $U$ and $f\in C_c(V)$, 
	there exists $\tilde{f}\in C_c(U)$ such that $\tilde {f}|_{V}=f$ by the Tietze extension theorem.
	Now, we obtain $\tilde{f}\in \mathcal{C}(G)$ such that $\tilde{f}|_{G_F}=f$.
	By Proposition \ref{supported on bisection}, $\mathcal{C}(G_F)$ is the linear span of $\bigcup_{U} C_c(G_F\cap U)$, where the union is taken over all open Hausdorff subsets $U\subset G$. 
	Therefore, the restriction $\mathcal{C}(G)\ni f \mapsto f|_{G_F}\in \mathcal{C}(G_F)$ is surjective.	
	
	By the universality of $C^*(G)$,
	the restriction $\mathcal{C}(G)\ni f \mapsto f|_{G_F}\in \mathcal{C}(G_F)$ extends to the *-homomorphism $C^*(G)\to C^*(G_F)$.
	Since the image of $C^*(G)$ is dense in $C^*(G_F)$, $C^*(G)\to C^*(G_F)$ is surjective.
	\qed
\end{proof}

\section{Quotients of \'etale groupoids}
In this section,
we introduce the notion of quotient \'etale groupoids.

\subsection{Quotients of \'etale groupoids}
\begin{defi}
	Let $G$ be a groupoid.
	A subgroupoid $H\subset G$ is said to be normal if
	\begin{enumerate}
		\item $G^{(0)}\subset H\subset \Iso(G)$ holds and
		\item $\alpha H\alpha^{-1}\subset H$ holds for all $\alpha\in G$.
	\end{enumerate} 
\end{defi}

\begin{defi}
	Let $G$ be a groupoid and $H\subset G$ be a normal subgroupoid.
	Then,
	we define an equivalence relation $\sim$ on $G$ by declaring that $\alpha\sim\beta$ if $s(\alpha)=s(\beta)$ and $\alpha\beta^{-1}\in H$.
	We denote the quotient set $G/{\sim}$ by $G/H$. 
\end{defi}

\begin{lem}\label{s and r of quotient}
	Let $G$ be a groupoid and $H\subset G$ be a normal subgroupoid.
	Suppose that $\alpha,\alpha'\in G$ satisfy $\alpha\sim\alpha'$.
	Then,
	we have $s(\alpha)=s(\alpha')$ and $r(\alpha)=r(\alpha')$.
\end{lem}

\begin{proof}
	It follows that $s(\alpha)=s(\alpha')$ from the definition of $\alpha\sim\alpha'$.
	Since $\alpha\alpha'^{-1}\in H\subset \Iso(G)$,
	we have $r(\alpha)=r(\alpha\alpha'^{-1})=s(\alpha\alpha'^{-1})=r(\alpha')$.\qed
\end{proof}

\begin{lem}\label{multi of quotient}
	Let $G$ be a groupoid and $H\subset G$ be a normal subgroupoid.
	Suppose that $\alpha,\alpha',\beta,\beta'\in G$ satisfy $\alpha\sim\alpha'$, $\beta\sim\beta'$, $s(\alpha)=r(\beta)$.
	Then,
	we have $s(\alpha')=r(\beta')$ and $\alpha\beta\sim \alpha'\beta'$.
\end{lem}

\begin{proof}
	By Lemma \ref{s and r of quotient}, we have $s(\alpha)=s(\alpha')$ and $r(\beta)=r(\beta')$.
	Using $s(\alpha)=r(\beta)$, we obtain $s(\alpha')=r(\beta')$.
	
	The last assertion follows from a simple calculation.
	Indeed,
	we have $s(\alpha\beta)=s(\beta)=s(\beta')=s(\alpha'\beta')$ and
	\[\alpha\beta(\alpha'\beta')^{-1}=\alpha\beta\beta'^{-1}\alpha'^{-1}=(\alpha\beta\beta'^{-1}\alpha^{-1})(\alpha\alpha'^{-1})\in H.
	\]
	Note that $\alpha\beta\beta'^{-1}\alpha^{-1}\in H$, since $H$ is normal.
	\qed
\end{proof}
\begin{defi}
	Let $G$ be a groupoid, $H\subset G$ be a normal subgroupoid and $q\colon G\to G/H$ be the quotient map.
	A groupoid structure of $G/H$ is defined as the following;
	\begin{itemize}
		\item a unit space $(G/H)^{(0)}$ is $q(G^{(0)})$, which can be identified with $G^{(0)}$ via an injection $q|_{G^{(0)}}$,
		\item source and range maps $s,r\colon G/H\to G^{(0)}$ are defined by $s(q(\gamma))\defeq q(s(\gamma))$, $r(q(\gamma))\defeq q(r(\gamma))$ for $\gamma\in G$,
		\item a multiplication of $G/H$ is defined by $q(\alpha)q(\beta)\defeq 
		q(\alpha\beta)$ for $\alpha,\beta\in G$ with $s(\alpha)=r(\beta)$.
	\end{itemize}
	One can see that the inverse map of $G/H$ satisfies $q(\gamma)^{-1}=q(\gamma^{-1})$ for $\gamma\in G$.
	Then,
	$G/H$ is a groupoid under these operations.
\end{defi}
\begin{rem}
	The operations of $G/H$ are well-defined by Lemma \ref{s and r of quotient} and Lemma \ref{multi of quotient}.
\end{rem}

If $G$ is a topological groupoid,
then we consider the quotient topology as a topology of $G/H$.

\begin{lem}\label{general quotient is local homeo}
	Let $G$ be an \'etale groupoid and $H\subset G$ be an open normal subgroupoid.
	Then,
	the quotient map $q\colon G\to G/H$ is an open map.
	Furthermore,
	$q$ is a local homeomorphism.
\end{lem}

\begin{proof}
	Let $U\subset G$ be an open subset.
	Then,
	$q^{-1}(q(U))=UH$ is an open subset of $G$ by Proposition \ref{multi}.
	Hence, $q(U)\subset G/H$ is an open subset by the definition of the quotient topology.
		
	Next,
	we show that the quotient map $q\colon G\to G/H$ is a local homeomorphism.
	Fix a $\gamma\in G$.
	Then,
	take an open bisection $U\subset G$ with $\gamma\in U$.
	One can see that $q|_U$ is injective.
	Since $q$ is an open map,
	$q|_U$ is a homeomorphism onto an open subset $q(U)\subset G$.
	Hence,
	$q$ is a local homeomorphism.
	\qed
\end{proof}

Observe that $q|_{G^{(0)}}\colon G^{(0)}\to (G/H)^{(0)}$ is homeomorphic.

\begin{prop}\label{quotient groupoid}
	Let $G$ be an \'etale groupoid and $H\subset G$ be an open normal subgroupoid.
	Then,
	$G/H$ is an \'etale groupoid.
\end{prop}

\begin{proof}
	First,
	we show the continuity of the inverse $G/H\ni \delta\mapsto\delta^{-1}\in G/H$.
	One can see that a map $G\ni \gamma\mapsto q(\gamma)^{-1}\in G/H$ is continuous,
	since the following diagram is commutative;
	
	\begin{center}
		\begin{tikzcd}
			& G \arrow[r,"\mathrm{inverse}"] \arrow[d, "q",swap]  & G \arrow[d,"q"]\\
			& G/H \arrow[r,"\mathrm{inverse}"]  &G/H.
		\end{tikzcd}
	\end{center}
	By the definition of the quotient topology,
	the inverse of $G/H$ is continuous.
	
	Next,
	we show that the multiplication of $G/H$ is continuous.
	Take $(q(\alpha),q(\beta))\in (G/H)^{(2)}$ and an open set $U\subset G/H$ such that $q(\alpha)q(\beta)\in U$.
	Since $\alpha\beta \in q^{-1}(U)$ and $q^{-1}(U)\subset G$ is open,
	there exist open sets $V_1,V_2\subset G$ such that $\alpha\in V_1, \beta\in V_2$ and $V_1V_2\subset q^{-1}(U)$.
	Subsets $V_1,V_2\subset G$ are open,
	so $q(V_1),q(V_2)\subset G/H$ are open.
	One can see that $q(\alpha)\in q(V_1)$, $q(\beta)\in q(V_2)$ and $q(V_1)q(V_2)=q(V_1V_2)\subset U$.
	Therefore,
	the multiplication of $G/H$ is continuous.	
	
	Finally,
	we show that $G/H$ is \'etale.
	Since the restriction $q|_{G^{(0)}}$ gives a homeomorphism from $G^{(0)}$ to $(G/H)^{(0)}$,
	$(G/H)^{(0)}$ is a locally compact Hausdorff space.
	One can see that the source map $s\colon G/H\to (G/H)^{(0)}$ is a local homeomorphism,
	since we have Lemma \ref{general quotient is local homeo} and the following diagram is commutative for every open bisection $U\subset G$;
	\begin{center}
		\begin{tikzcd}
			& U \arrow[r,"q"] \arrow[d, "s",swap]  & q(U) \arrow[d,"s"]\\
			& s(U) \arrow[r,"q"]  & s(q(U)).
		\end{tikzcd}
	\end{center}
	Therefore, $G/H$ is an \'etale groupoid.
	\qed
\end{proof}

\begin{defi}
	Let $G_1$ and $G_2$ be groupoids.
	A map $\Phi\colon G_1\to G_2$ is called a groupoid homomorphism if  $(\Phi(\alpha),\Phi(\beta))\in G_2^{(2)}$ and $\Phi(\alpha\beta)=\Phi(\alpha)\Phi(\beta)$ hold for all $(\alpha,\beta)\in G^{(2)}_1$.
\end{defi}

Now, we obtain the next theorem by Lemma \ref{general quotient is local homeo} and Proposition \ref{quotient groupoid},

\begin{thm}\label{quotient groupoid main theorem}
	Let $G$ be an \'etale groupoid and $H\subset G$ be an open normal subgroupoid.
	Then,
	the sequence of the \'etale groupoids
	\begin{center}
		\begin{tikzpicture}[auto]
		\node (H) at (-3,0) {$H$};\node (G) at (0,0) {$G$};\node (G/H) at (3,0) {$G/H$}; 
		\draw[right hook->] (H) to node {\scriptsize inclusion} (G);
		\draw[->>] (G) to node {$\scriptstyle q$} (G/H);
		\end{tikzpicture}
	\end{center}
	is exact, that is, $q^{-1}((G/H)^{(0)})=H$.
\end{thm}

\begin{prop}\label{Hausdorffness of G/H}
	Let $G$ be an \'etale groupoid and $H\subset G$ be an open normal subgroupoid.
	Then, $G/H$ is Hausdorff if and only if $H\subset G$ is closed.
\end{prop}

\begin{proof}
	Recall that an \'etale groupoid $G$ is Hausdorff if and only if its unit space $G^{(0)}$ is a closed subset of $G$ (for example, see \cite[Lemma 2.3.2]{asims}). 
	If $G/H$ is Hausdorff,
	$(G/H)^{(0)}\subset G/H$ is closed.
	Hence, $H=q^{-1}((G/H)^{(0)})$ is a closed subset of $G$.
	
	Suppose that $H\subset G$ is closed.
	Since $q$ is an open map,
	$(G/H)\setminus (G/H)^{(0)} =q(G\setminus H)\subset G/H$ is open.
	Hence,
	$(G/H)^{(0)}\subset G/H$ is closed, which implies that $G/H$ is Hausdorff.\qed
\end{proof}

\begin{prop}\label{Iso circ is normal sub}
	Let $G$ be an \'etale groupoid.
	Then,
	the interior of isotropy $\Iso(G)^{\circ}\subset \Iso(G)$ is a normal subgroupoid.
\end{prop}

\begin{proof}
	We show that $\Iso(G)^{\circ}$ is normal.
	By Proposition \ref{local homeo}, $G^{(0)}$ is contained in $\Iso(G)^{\circ}$.
	Take $\alpha\in G$ and $\gamma\in \Iso(G)^{\circ}$ with $s(\alpha)=r(\gamma)$.
	There exist open bisections $U,V\subset G$ with $\alpha\in U$ and $\gamma\in V\subset\Iso(G)$.
	Then, by Proposition \ref{multi},
	$UVU^{-1}\subset G$ is an open subset which contains $\alpha\gamma\alpha^{-1}$.  
	Since $U$ is bisection and $V\subset \Iso(G)$,
	we have $UVU^{-1}\subset \Iso(G)$.
	Therefore,
	$\alpha\gamma\alpha^{-1}\in\Iso(G)^{\circ}$ and $\Iso(G)^{\circ}$ is an open normal subgroupoid.
	\qed
\end{proof}

An \'etale groupoid $G/ \Iso(G)^{\circ}$,
which is a special case of quotient groupoids,
coincides with a groupoid of germs of the canonical action (see \cite[Section 3]{renault}).
One can see that $G/\Iso(G)^{\circ}$ is effective.

\subsection{*-homomorphisms induced by quotients of \'etale groupoids}

For an \'etale groupoid $G$ and an open normal subgroupoid $H\subset G$,
we have obtained the quotient \'etale groupoid $G/H$.
Next, we see that the quotient map $q\colon G\to G/H$ induces a *-homomorphism $C^*(G)\to C^*(G/H)$.

For $f\in \mathcal{C}(G)$,
we define $\tilde{f}\colon G/H\to \C$ by
\[
\tilde{f}(\gamma)\defeq \sum_{q(\alpha)=\gamma}f(\alpha)
\]
for $\gamma\in G/H$.
Then,
the following proposition holds.
\begin{prop}\label{quotient induces *-hom}
	Let $G$ be an \'etale groupoid and $H\subset G$ be an open normal subgroupoid.
	Then,
	$\mathcal{C}(G)\ni f \mapsto \tilde{f}\in \mathcal{C}(G/H)$ is a surjective *-homomorphism.
\end{prop}

\begin{proof}
	First,
	we show $\tilde{f}\in \mathcal{C}(G/H)$.
	We may assume that there exists an open bisection $U\subset G$ such that $f|_U\in C_c(U)$ and $f|_{G\setminus U}=0$.
	Then,
	$q(U)\subset G/H$ is an open bisection and $\tilde{f}|_{q(U)}=f\circ (q|_U)^{-1}\in C_c(q(U))$,
	since $q|_U$ is a homeomorphism onto the image.
	Moreover,
	one can see that $\tilde{f}_{(G/H)\setminus q(U)}=0$.
	Hence, $\tilde{f}\in C_c(q(U))\subset \mathcal{C}(G/H)$.
	
	We show that $\mathcal{C}(G)\ni f\mapsto \tilde{f}\in \mathcal{C}(G/H)$ is a *-homomorphism. 
	We only check that $\mathcal{C}(G)\ni f\mapsto \tilde{f}\in \mathcal{C}(G/H)$ preserves the multiplications,
	since it is easy to check that this map is linear and preserves the involutions.
	For all $f,g\in \mathcal{C}(G)$ and $\gamma'\in G/H$, we have 
	\begin{align*}	\widetilde{f*g}(\gamma')&=\sum_{q(\gamma)=\gamma'}f*g(\gamma)=\sum_{q(\gamma)=\gamma'}\sum_{\alpha\beta=\gamma}f(\alpha)g(\beta)=\sum_{q(\alpha\beta)=\gamma'}f(\alpha)g(\beta), \\
	\tilde{f}*\tilde{g}(\gamma')&=\sum_{\alpha'\beta'=\gamma'}\tilde{f}(\alpha')\tilde{g}(\beta')=\sum_{\alpha'\beta'=\gamma'}\sum_{q(\alpha)=\alpha'}\sum_{q(\beta)=\beta'}f(\alpha)g(\beta) \\
	&=\sum_{q(\alpha\beta)=\gamma'}f(\alpha)g(\beta).
	\end{align*}
	
	Finally, we show that $\mathcal{C}(G)\ni f\mapsto \tilde{f}\in \mathcal{C}(G/H)$ is surjective.
	Note that $\{q(U)\subset G/H\mid \text{$U\subset G$ is an open bisection}\}$ is an open basis of $G$.
	Let $U\subset G$ be an open bisection and $f\in C_c(q(U))$.
	One can see that $q|_U$ is a homeomorphism onto its image.
	Define $g\defeq f\circ q|_U\in C_c(U)$. 
	Then, we have $\widetilde{g}=f$.	
	By Proposition \ref{canonical surjection},
	$\mathcal{C}(G/H)$ is the linear span of $\bigcup_U C_c(q(U))$,
	where the union is taken over all open bisections $U\subset G$.
	Hence,
	$\mathcal{C}(G)\ni f\mapsto \tilde{f}\in \mathcal{C}(G/H)$ is a surjective *-homomorphism.
	\qed
\end{proof}

By the universality of $C^*(G)$,
the surjective *-homomorphism $\mathcal{C}(G)\ni f \mapsto \tilde{f}\in \mathcal{C}(G/H)$ induces the surjective *-homomorphism $Q\colon C^*(G)\to C^*(G/H)$.

\begin{prop}\label{Q is inj on masa}
	Let $Q\colon C^*(G)\to C^*(G/H)$ be the *-homomorphism as above.
	Then, $\ker Q\cap C_0(G^{(0)})=\{0\}$ holds. 
\end{prop}

\begin{proof}
	Since the universal norm of a function in $C_c(G^{(0)})$ coincides with the supremum norm,
	$Q|_{C_c(G^{(0)})}$ is isometric.
	Therefore,
	$Q|_{C_0(G^{(0)})}$ is isometric and $\ker Q\cap C_0(G^{(0)})=\{0\}$.
	\qed
\end{proof}

\begin{defi}
Let $G$ be an \'etale groupoid.
We say that $G$ has the full uniqueness property if every non-zero ideal $I\subset C^*(G)$ satisfy $I\cap C_0(G^{(0)})\not=\{0\}$.
\end{defi}

The full uniqueness property of $G$ is equivalent to the condition that a *-representation $\pi\colon C^*(G)\to B(H)$ is injective if and only if $\pi|_{C_0(G^{(0)})}$ is injective.  

\begin{lem}\label{QC*G to C*G/H}
	Let $G$ be an \'etale groupoid and $H\subset G$ be an open normal subgroupoid.
	Then,
	the *-homomorphism $Q\colon C^*(G)\to C^*(G/H)$ induced by Proposition \ref{quotient induces *-hom} is injective if and only if $H= G^{(0)}$.
\end{lem}

\begin{proof}
	It is clear that the *-homomorphism $Q\colon C^*(G)\to C^*(G/H)$ is injective if $H= G^{(0)}$.
	Suppose that $G^{(0)}\subsetneq H$ and take $\gamma_0\in H\setminus G^{(0)}$.
	There exists an open bisection $U\subset G$ with $\gamma_0\in U\subset H$.
	By the Urysohn lemma,
	there exists $f_1\in C_c(U)$ with $f_1(\gamma_0)=1$.
	Define $f_2\in C_c(G^{(0)})$ by 
	\begin{align*}
	f_2(\gamma)=
	\begin{cases}
	f_1\circ (s|_U)^{-1}(\gamma) & (\gamma\in s(U))\\
	0 & ( \gamma\in G^{(0)}\setminus s(U)).
	\end{cases}
	\end{align*}
	We have $f\defeq f_1-f_2\not=0$, since $f(\gamma_0)=1$.
	One can see that $Q(f)=0$,
	which means that $Q$ is not injective.\qed	
\end{proof}

Recall that an \'etale groupoid $G$ is said to be effective if $G^{(0)}=\Iso(G)^{\circ}$.
\begin{cor}[cf.\ {\cite[Proposition 5.5]{Brown2014}}]\label{intersection prpoerty implies effective}\label{uniqueness property implies effectiveness}
Let $G$ be an \'etale groupoid.
If $G$ has the full uniqueness property,
then $G$ is effective.
\end{cor}

\begin{proof}
By Proposition \ref{Iso circ is normal sub},
$\Iso(G)^{\circ}$ is a normal subgroupoid of $G$.
Letting $Q\colon C^*(G)\to C^*(G/\Iso(G)^{\circ})$ be the *-homomorphism induced by Proposition \ref{quotient induces *-hom},
we have $\ker Q\cap C_0(G^{(0)})=\{0\}$.
The full uniqueness property implies that $Q\colon C^*(G)\to C^*(G/\Iso(G)^{\circ})$ is injective.
Therefore,
we obtain  $\Iso(G)^{\circ}=G^{(0)}$ by Lemma \ref{QC*G to C*G/H}.\qed
\end{proof}

\begin{rem}
	It had been known that Corollary \ref{intersection prpoerty implies effective} holds for Hausdorff \'etale groupoids as proved in \cite[Proposition 5.5]{Brown2014}. 
	In \cite[Proposition 5.5]{Brown2014},
	the authors use the augmentation representations,
	which seems to work for non-Hausdorff \'etale groupoids. 
	
	As shown in Proposition \ref{quotient induces *-hom},
	the quotient map $G\to G/\Iso(G)^{\circ}$ of \'etale groupoids induces the *-homomorphism $C^*(G)\to C^*(G/\Iso(G)^{\circ})$.  
	Using this *-homomorphism,
	we obtain the proof of Corollary \ref{intersection prpoerty implies effective}, which seems to be more direct than the one in \cite[Proposition 5.5]{Brown2014}.
	We shall remark that $G/\Iso(G)^{\circ}$ coincides with the groupoid of germs of the canonical action in \cite[Proposition 3.2]{renault}.

	The converse of Corollary \ref{intersection prpoerty implies effective} does not hold for non-Hausdorff \'etale groupoids.
	Indeed,
	Exel showed that there exists an effective \'etale groupoid $G$ which does not have the full uniqueness property in \cite{ExelnonHausdorff} (cf.\ Example \ref{Exelgroupoid}).
\end{rem}

\begin{ex}[{\cite[Section 2]{ExelnonHausdorff}}]\label{Exelgroupoid}
	Let $X\defeq ([-1,1]\times \{0\})\cup (\{0\}\times [-1,1])\subset\R^2$ and $K\defeq\{e,s,t,st\}$ be the Klein group,
	which is isomorphic to $\Z/2\Z\oplus\Z/2\Z$.
	We define an action $\sigma$ of $K$ on $X$ by
	\begin{align*}
	\sigma_s((x,y))=(-x,y), \ \  \sigma_t((x,y))=(x,-y), \ \  \sigma_{st}((x,y))=(-x,-y)
	\end{align*}
	fot $(x,y)\in X$.
	
	\begin{figure}
		\centering
		\begin{tikzpicture}[auto]
		\node(01) at (0,1){};
		\node (-10) at (-1,0){};
		\node(0) at (0,0){};
		\node(10) at (1,0){};
		\node(0-1) at (0,-1){};
		\fill[black] (01) circle (1pt);
		\fill[black] (10) circle (1pt);
		\fill[black] (-10) circle (1pt); \fill[black] (0-1) circle (1pt);
		\draw[-] (0,-1) to node {} (0,1);
		\draw[-] (-1,0) to node {} (1,0);
		\node at (0,1.3) {$\scriptstyle (0,1)$};
		\node at (1,-0.2) {$\scriptstyle (1,0)$};
		\node at (-1,-0.2) {$\scriptstyle (-1,0)$};
		\node at (0,-1.2) {$\scriptstyle (0,-1)$};
		\end{tikzpicture}
		\caption{Picture of $X$ in Example \ref{Exelgroupoid}}
	\end{figure}
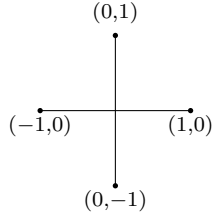
	
	Consider the transformation groupoid $G\defeq K\ltimes_{\sigma}X$ (see Example \ref{action}).
	One can see that 
	\begin{align*}
	\Iso(G)=G^{(0)} &\cup \{(s,(0,y))\in G\mid y\in [-1,1]\}  \\
	&\cup \{(t,(x,0))\in G\mid x\in [-1,1] \}\cup \{(st,(0,0))\}.
	\end{align*}
	Moreover, we have $\Iso(G)^{\circ}= \Iso(G)\setminus \{(s,(0,0)),(t,(0,0)),(st,(0,0))\}$.
	Since $\Iso(G)^{\circ}$ is not closed in $G$ (for example, $(s,(0,0))\in \overline{\Iso(G)^{\circ}}\setminus \Iso(G)^{\circ}$),
	the quotient \'etale groupoid $G/\Iso(G)^{\circ}$ is not Hausdorff by Proposition \ref{Hausdorffness of G/H}.
	In \cite{ExelnonHausdorff},
	Exel shows that $G/\Iso(G)^{\circ}$ does not have the full uniqueness property,
	although it is effective. 
\end{ex}

Based on the work in \cite{LisaRuyAidan}, we shall give a necessary and sufficient condition of the full uniqueness property.
Let $G$ be an \'etale groupoid.
We denote the left representation at $x\in G^{(0)}$ by $\lambda_x\colon \mathcal{C}(G)\to B(\ell^2(G_x))$.
By the universality of $C^*(G)$,
the left representation extends to the *-representation $\lambda_x\colon C^*(G)\to B(\ell^2(G_x))$.
Following \cite{LisaRuyAidan}, we say that an element $a\in C^*(G)$ is singular if the interior of $\{\gamma\in G\mid \i<\delta_{\gamma}|\lambda_{s(\gamma)}(a)\delta_{s(\gamma)}>\not=0\}$ is empty,
where $\delta_{\gamma}\in \ell^2(G_x)$ denotes the delta function at $\gamma\in G_x$.

We denote the reduced groupoid C*-algebra of $G$ by $C^*_{\lambda}(G)$.
We denote the canonical surjective *-homomorphism by $Q\colon C^*(G)\to C^*_{\lambda}(G)$.
In \cite{LisaRuyAidan}, the notion of a singular element is defined for elements in $C^*_{\lambda}(G)$ in the same way as elements in $C^*(G)$.
Since an element in $\ker Q$ is singular,
$C^*(G)$ has no nonzero singular elements if and only if $C^*_{\lambda}(G)$ has no nonzero singular elements and $C^*(G)\simeq C^*_{\lambda}(G)$ via the canonical *-homomorphism $Q$.

\begin{prop}
Let $G$ be a second countable \'etale groupoid.
Then,
$G$ has the full uniqueness property if and only if
\begin{itemize}
	\item $G$ is effective and
	\item $C^*(G)$ has no nonzero singular elements.
\end{itemize}
\end{prop}

\begin{proof}
	Assume that $G$ has the full uniqueness property.
	Corollary \ref{intersection prpoerty implies effective} implies that $G$ is effective.
	We show that $C^*(G)$ has no nonzero singular elements.
	Let $a\in C^*(G)$ be a singular element.
	We define $S\defeq \{x\in G^{(0)}\mid G_x\cap G^x=\{x\}\}$,
	which is a dense subset of $G^{(0)}$ by \cite[Proposition 3.6]{renault}.
	Therefore, letting $\pi\defeq \bigoplus_{x\in S}\lambda_x$,
	$\pi$ is injective on $C_0(G^{(0)})$.
	The full uniqueness property implies that $\pi$ is injective.
	Since $\lambda_x(a)=\lambda_x(Q(a))$ holds for all $x\in G^{(0)}$,
	we have 
	\[
	s(\{ \gamma\in G\mid \i<\delta_{\gamma}|\lambda_{s(\gamma)}(a)\delta_{s(\gamma)}>\not=0 \})\subset G^{(0)}\setminus S
	\]
	by \cite[Lemma 4.2]{LisaRuyAidan}.
	Using this fact, one can see that $\pi(a)=0$ if $a\in C^*(G)$ is singular.
	Hence, $C^*(G)$ has no nonzero singular element. 
	
	Next, we show the converse.
	Note that $Q$ is an isomorphism by the assumpution that $C^*(G)$ has no nonzero singular elements.
	Now, the full uniqueness property of $G$ follows from \cite[Theorem 4.4]{LisaRuyAidan}.
	\qed
\end{proof}

\section{The abelianizations of \'etale groupoid C*-algebras}
In this section,
we calculate the abelianizations of \'etale groupoid C*-algebras.
First, recall the abelianizations of C*-algebras.
For a C*-algebra $A$,
its abelianization is defined by $A^{\ab}=A/I$,
where $I\subset A$ is the closed two-sided ideal generated by $\{xy-yx\in A\mid x,y\in A\}$. 
The abelianization $A^{\ab}$ is a commutative C*-algebra with the following universality;
for all commutative C*-algebra $B$ and *-homomorphism $\pi\colon A\to B$,
there exists the unique *-homomorphism $\tilde{\pi}\colon A^{\ab}\to B$ such that $\tilde{\pi}\circ q=\pi$,
where $q\colon A\to A^{\ab}$ denotes the quotient map.

\subsection{One dimensional representations of a groupoid C*-algebra}

For a C*-algebra $A$,
we denote the set of all one dimensional nondegenerate representations of $A$ by $\Delta(A)$.
Namely,
$\Delta(A)$ is the set of all nonzero *-homomorphisms from $A$ to $\C$.
We suppose that $\Delta(A)$ is equipped with the pointwise convergence topology.
If $A$ is commutative,
$\Delta(A)$ is known as the Gelfand spectrum of $A$.
First, we calculate $\Delta(C^*(G))$.

Let $G$ be an \'etale groupoid and $x\in G^{(0)}$ be a fixed point of $G$.
Note that $G_x$ is a discrete group.
We denote the surjection in Proposition \ref{canonical surjection} by $Q_x\colon C^*(G)\to C^*(G_x)$ temporarily.
Also, we denote the circle group by $\T\defeq\{z\in\C\mid \lvert z\rvert=1\}$. 
For a group homomorphism $\chi\colon G_x\to\T$,
a map $C_c(G_x)\ni f\mapsto \sum_{\gamma\in G_x}\chi(\gamma)f(\gamma)\in \C$ is a *-homomorphism.
This *-homomorphism extends to the *-homomorphism $C^*(G_x)\to \C$,
which we also denote by $\chi\colon C^*(G_x)\to \C$.   

\begin{defi}\label{def of phix,chi}
	Let $G$ be an \'etale groupoid, $x\in G^{(0)}$ be a fixed point and $\chi\colon G_x\to \T$ be a group homomorphism.
	Then,
	we define a *-homomorphism $\varphi_{x,\chi}\colon C^*(G)\to \C$ by $\varphi_{x,\chi}\defeq \chi\circ Q_x$.
\end{defi}
We will show that all elements of $\Delta(C^*(G))$ are obtained by this form (Theorem \ref{charaC*G}).

\begin{prop}\label{def of fixed pt}
	Let $G$ be an \'etale groupoid and $\varphi\in\chara(C^*(G))$.
	Then,
	there uniquely exists $x_{\varphi}\in G^{(0)}$ which satisfies $\varphi(f)=f(x_{\varphi})$ for all $f\in C_0(G^{(0)})$.
\end{prop}

\begin{proof}
	We have $\varphi|_{C_0(G^{0})}\not=0$ since $C_0(G^{(0)})$ has an approximate identity of $C^*(G)$.
	Therefore, $\varphi|_{C_0(G^{(0)})}$ belongs to $\Delta(C_0(G^{(0)}))$.
	Now, the existence and uniqueness of $x_{\varphi}\in G^{(0)}$ follow from the Gelfand-Naimark theorem. 
	\qed 
\end{proof}

\begin{prop}
	Let $G$ be an \'etale groupoid and $\varphi\in\chara(C^*(G))$.
	Then,
	$x_{\varphi}\in G^{(0)}$ is a fixed point.
\end{prop}

\begin{proof}
	Assume that $\gamma\in G$ satisfies $s(\gamma)=x_{\varphi}$.
	We will show $r(\gamma)=x_{\varphi}$.
	There exists an open bisection $U\subset G$ with $\gamma\in U$.
	Take $n_{\gamma}\in C_c(U)$ which satisfies $n_{\gamma}(\gamma)=1$.
	Note that we have $n_{\gamma}^**n_{\gamma}\in C_c(G^{(0)})$ and $n_{\gamma}^**n_{\gamma}(x_{\varphi})=\lvert n_{\gamma}(\gamma)\rvert^2=1$.
	Fix $f\in C_c(G^{(0)})$ arbitrarily.

	Direct calculations show that $n^*_{\gamma}*f*n_{\gamma}(x_{\varphi})=\overline{n_{\gamma}(\gamma)}f(r(\gamma))n_{\gamma}(\gamma)=f(r(\gamma))$. 
	On the other hand,
	one can see that $n_{\gamma}^**f*n_{\gamma}\in C_c(G^{(0)})$.
	Then, we have
	\[
	n^*_{\gamma}*f*n_{\gamma}(x_{\varphi})=\varphi(n^*_{\gamma}*f*n_{\gamma})=\varphi(n^*_{\gamma})\varphi(f)\varphi(n_{\gamma})=\varphi(n^*_{\gamma}*n_{\gamma})\varphi(f)=f(x_{\varphi}).
	\]
	
	Therefore,
	$f(r(\gamma))=f(x_{\varphi})$ holds for all $f\in C_c(G^{(0)})$,
	which implies $r(\gamma)=x_{\varphi}$.
	Hence, $x_{\varphi}\in G^{(0)}$ is a fixed point of $G$.\qed
\end{proof}

\begin{prop}\label{f_gamma}
	Let $G$ be an \'etale groupoid, $\varphi\in\chara(C^*(G))$ and $\gamma\in G_{x_{\varphi}}$.
	Take an open bisection $U_{\gamma}\subset G$ with $\gamma\in U_{\gamma}$ and $f_{\gamma}\in C_c(U_{\gamma})$ with $f_{\gamma}(\gamma)=1$.
	Then,
	$\varphi(f_{\gamma})$ is independent of the choice of $U_{\gamma}$ and $f_{\gamma}$. 
	Moreover,
	we have $\varphi(f_{\gamma})\in\T$.
\end{prop}

\begin{proof}
	First, we show $\varphi(f_{\gamma})\in\T$.
	Since $f_{\gamma}^**f_{\gamma}\in C_0(G^{(0)})$,
	we have
	\[
	\lvert\varphi(f_{\gamma})\rvert^2=\varphi(f_{\gamma}^**f_{\gamma})=f_{\gamma}^**f_{\gamma}(x_{\varphi})=\lvert f_{\gamma}(\gamma)\rvert^2=1.
	\]
	Therefore,
	$\varphi(f_{\gamma})\in \T$.
	
	Second,
	we show that $\varphi(f_{\gamma})$ is independent of the choice of $U_{\gamma}$ and $f_{\gamma}$.
	Assume that $f_{\gamma}\in C_c(U_{\gamma})$ and $g_{\gamma}\in C_c(V_{\gamma})$  satisfies $f_{\gamma}(\gamma)=g_{\gamma}(\gamma)=1$,
	where $U_{\gamma}$ and $V_{\gamma}\subset G$ are open bisections.
	Find a function $h\in C_c(s(U_{\gamma}\cap V_{\gamma}))\subset C_c(G^{(0)})$ such that $h(s(\gamma))=1$.
	Recall that $s(\gamma)=r(\gamma)=x_{\varphi}$ since $x_{\varphi}$ is a fixed point.
	Also, note that $\varphi(h)=h(x_{\varphi})=1$.
	Putting $\widetilde{f_{\gamma}}\defeq f_{\gamma}*h$ and $\widetilde{g_{\gamma}}=g_{\gamma}*h$,
	we have that $\widetilde{f_{\gamma}}$ and $\widetilde{g_{\gamma}}$ are contained in $C_c(U_{\gamma}\cap V_{\gamma})$.
	Then,
	it follows that $\widetilde{f_{\gamma}}^**\widetilde{g_{\gamma}}\in C_0(G^{(0)})$ and
	\begin{align*}
	\overline{\varphi(f_{\gamma})}\varphi(g_{\gamma})&=\overline{\varphi(h)\varphi(f_{\gamma})}\varphi(g_{\gamma})\varphi(h)=\varphi(\widetilde{f_{\gamma}}^**\widetilde{g_{\gamma}}) \\
	&=\widetilde{f_{\gamma}}^**\widetilde{g_{\gamma}}(x_{\varphi})=\overline{h(r(\gamma))f_{\gamma}(\gamma)}g_{\gamma}(\gamma)h(s(\gamma))=1.
	\end{align*}
	Now,
	we have $\varphi(f_{\gamma})=\varphi(g_{\gamma})$ since $\varphi(f_{\gamma})\in \T$.\qed
	
\end{proof}

\begin{prop}\label{def of chi}
	Let $G$ be an \'etale groupoid and $\varphi\in\chara(C^*(G))$.
	We define $\chi_{\varphi}\colon G_{x_{\varphi}}\to \T$ by $\chi_{\varphi}(\gamma)\defeq \varphi(f_{\gamma})$,
	where $\gamma\in G_{x_{\varphi}}$ and $f_{\gamma}\in \mathcal{C}(G)$ is a function as in Proposition \ref{f_gamma}.
	Then,
	$\chi_{\varphi}\colon G_{x_{\varphi}}\to \T$ is a group homomorphism. 
\end{prop}

\begin{proof}
	Take $\alpha,\beta\in G_{x_{\varphi}}$.
	We show $\chi_{\varphi}(\alpha)\chi_{\varphi}(\beta)=\chi_{\varphi}(\alpha\beta)$.
	Take $f_{\alpha},f_{\beta}\in \mathcal{C}(G)$ as in Proposition \ref{f_gamma}.
	It follows that $f_{\alpha}*f_{\beta}\in C_c(U)$ for some open bisection $U\subset G$ and $f_{\alpha}*f_{\beta}(\alpha\beta)=1$.
	Hence,
	we have
	\[
	\chi_{\varphi}(\alpha\beta)=\varphi(f_{\alpha}*f_{\beta})=\varphi(f_{\alpha})\varphi(f_{\beta})=\chi_{\varphi}(\alpha)\chi_{\varphi}(\beta)
	\]
	by the definition of $\chi_{\varphi}$.\qed
\end{proof}

\begin{prop}\label{katahou1}  
	Let $G$ be an \'etale groupoid.
	Then, we have $\varphi=\varphi_{x_{\varphi},\chi_{\varphi}}$ for all $\varphi\in\chara(C^*(G))$.
\end{prop}

\begin{proof}
	Take $f\in C_c(U)$,
	where $U\subset G$ is an open bisection.
	It suffices to show $\varphi(f)=\varphi_{x_{\varphi},\chi_{\varphi}}(f)$,
	since $C^*(G)$ is generated by such functions.
	Note that $f^**f\in C_c(G^{(0)})$.
	If $G_{x_{\varphi}}\cap f^{-1}(\C\setminus\{0\})=\emptyset$,
	then we have $0=f^**f(x_{\varphi})=\lvert \varphi(f)\rvert^2$.
	Since the restriction of $f|_{G_{x_{\varphi}}}$ is zero,
	it follows that $\varphi_{x_{\varphi},\chi_{\varphi}}(f)=0=\varphi(f)$.
	If $G_{x_{\varphi}}\cap f^{-1}(\C\setminus\{0\})\not=\emptyset$,
	$G_{x_{\varphi}}\cap f^{-1}(\C\setminus\{0\})$ is a singleton because $f$ is supported on an open bisection.
	Let $\gamma\in G_{x_{\varphi}}\cap f^{-1}(\C\setminus\{0\})$ be the unique element of $G_{x_{\varphi}}\cap f^{-1}(\C\setminus\{0\})$.
	Observe that $F\defeq f/f(\gamma)\in C_c(U)$ satisfies $F(\gamma)=1$.
	Now,
	we have
	\[
	\varphi_{x_{\varphi},\chi_{\varphi}}(f)=f(\gamma)\chi_{\varphi}(\gamma)=f(\gamma)\varphi(F)=\varphi(f). 
	\]
	Hence,
	we have $\varphi_{x_{\varphi},\chi_{\varphi}}=\varphi$.\qed
\end{proof}

\begin{prop}\label{katahou2} 
	Let $G$ be an \'etale groupoid,
	$x\in G^{(0)}$ be a fixed point and $\chi\colon G_x\to \T$ be a group homomorphism.
	Then,
	$x=x_{\varphi_{x,\chi}}$ and $\chi=\chi_{\varphi_{x,\chi}}$.
\end{prop}

\begin{proof}
	First, we show $x=x_{\varphi_{x,\chi}}$.
	Take $f\in C_c(G^{(0)})$ arbitrarily.
	Then,
	we have
	\[
	f(x_{\varphi_{x,\chi}})=\varphi_{x,\chi}(f)=f(x)\chi(x)=f(x).
	\]
	Hence, it follows $x=x_{\varphi_x}$.
	
	Next,
	we show $\chi=\chi_{\varphi_{x,\chi}}$.
	Take $\gamma\in G_x$ arbitrarily.
	There exist an open bisection $U\subset G$ with $\gamma\in U$ and $f\in C_c(U)$ with $f(\gamma)=1$.
	Then,
	we have
	\[
	\chi_{\varphi_{x,\chi}}(\gamma)=\varphi_{x,\chi}(f)=f(\gamma)\chi(\gamma)=\chi(\gamma).
	\]
	Hence,
	we have shown $x=x_{\varphi_{x,\chi}}$ and $\chi=\chi_{\varphi_{x,\chi}}$.\qed
\end{proof}

Combining Proposition \ref{katahou1} and Proposition \ref{katahou2},
we obtain the next theorem.

\begin{thm}\label{charaC*G} 
	Let $G$ be an \'etale groupoid.
	Define a set
	\begin{align*}
	\mathcal{D}\defeq \{(x,\chi)\mid x\in G^{(0)} & \text{ is a fixed point} \\
		 & \text{and $\chi\colon G_x\to\T$ is a group homomorphism} \}.
	\end{align*}
	Then,
	a map 
	\[
	\mathcal{D} \ni (x,\chi) \longrightarrow \varphi_{x,\chi}\in \chara(C^*(G))
	\]
	is bijective.
\end{thm}

\subsection{Construction of an \'etale abelian group bundle $G^{\ab}$}
For an \'etale groupoid $G$,
we construct an \'etale abelian group bundle $G^{\ab}$ so that $C^*(G)^{\ab}\simeq C^*(G^{\ab})$ holds.

\begin{prop}
Let $G$ be an \'etale group bundle.
We define the commutator subgroupoid of $G$ by $[G,G]\defeq \bigcup_{x\in G^{(0)}}[G_x,G_x]$,
where $[G_x,G_x]$ is the commutator subgroup of $G_x$.
Then,
$[G,G]$ is an open normal subgroupoid of $G$.
\end{prop}

\begin{proof}
	It is obvious that $[G,G]\subset G$ is a normal subgroupoid.
	We show that $[G,G]\subset G$ is open.
	Take $\gamma\in [G,G]$.
	By the definition of the commutator subgroup,
	there exists $\{\alpha_j\}_{j=1}^k,\{\beta_j\}_{j=1}^k\subset G_{s(\gamma)}$ such that 
	\[
	\gamma=\alpha_1\beta_1 \alpha_1^{-1}\beta_1^{-1}\alpha_2\beta_2 \alpha_2^{-1}\beta_2^{-1}\cdots\alpha_k\beta_k \alpha_k^{-1}\beta_k^{-1}.\]
	Take open bisections $U_j,V_j\subset G$ such that $\alpha_j\in U_j$ and $\beta_j\in V_j$ for all $j=1,2,\dots,k$.
	We show that $U_1V_1U_1^{-1}V_1^{-1}\subset [G,G]$,
	where we define $U^{-1}\defeq\{\gamma^{-1}\mid \gamma\in U\}$ for $U\subset G$.
	Fix $\gamma'\in U_1V_1U_1^{-1}V_1^{-1}$.
	Then,
	there exist $\alpha,\alpha'\in U_1$ and $\beta,\beta'\in V_1$ which satisfy $\gamma=\alpha\beta\alpha'^{-1}\beta'^{-1}$.
	Since $G$ is a group bundle,
	we have $s(\alpha)=s(\alpha')=s(\beta)=s(\beta')$.
	We obtain $\alpha=\alpha'$ and $\beta=\beta'$ because $U_1$ and $V_1$ are bijections.
	Therefore,
	$\gamma'=\alpha\beta\alpha^{-1}\beta^{-1}\in [G,G]$.
	Similarly, one can show that $U_1V_1U_1^{-1}V_1^{-1}U_2V_2U_2^{-1}V_2^{-1}\cdots U_kV_kU_k^{-1}V_k^{-1}\subset [G,G]$.
	
	By Proposition \ref{multi}, $U_1V_1U_1^{-1}V_1^{-1}U_2V_2U_2^{-1}V_2^{-1}\cdots U_kV_kU_k^{-1}V_k^{-1}$ is an open set and this contains $\gamma$.
	Hence,
	$[G,G]\subset G$ is an open normal subgroupoid.
	\qed
\end{proof}

Let $G$ be an \'etale groupoid.
Recall that the set of all fixed points $F\subset G^{(0)}$ is a closed subset of $G^{(0)}$ (Proposition \ref{fixed points are closed}).
We define $G_{\fix}\defeq G_F$,
which is an \'etale groupoid from Proposition \ref{restriction to inv is etale}.
Since we have $G_{\fix}=\Iso(G_{\fix})$,
$G_{\fix}$ is an \'etale group bundle.

\begin{defi}
	Let $G$ be an \'etale groupoid.
	We define the abelianization of $G$ by $G^{\ab}\defeq G_{\fix}/[G_{\fix},G_{\fix}]$.
\end{defi}

Let $G$ be an \'etale groupoid.
Then, we have a *-homomorphism $C^*(G)\to C^*(G_{\fix})$ induced by the restriction (Proposition \ref{canonical surjection}).
Composing with the *-homomorphism $C^*(G_{\fix})\to C^*(G^{\ab})$ in Proposition \ref{quotient induces *-hom},
we obtain a *-homomorphism $\pi\colon C^*(G)\to C^*(G^{\ab})$.

Note that $C^*(G)$ is commutative if and only if $G$ is an \'etale abelian group bundle.
In particular, $C^*(G^{\ab})$ is commutative.

\begin{lem}\label{character of CGab}
	Let $G$ be an \'etale groupoid.
	Then,
	a map $\Phi\colon \chara(C^*(G^{\ab}))\ni \chi\mapsto \chi\circ \pi\in \chara(C^*(G))$ is bijective.
\end{lem}

\begin{proof}
	Surjectivity of $\pi$ implies that $\Phi$ is injective.
	We show that $\Phi$ is surjective.
	Take $\varphi\in\chara(C^*(G))$.
	Then,
	we have the fixed point $x_{\varphi}\in G^{(0)}$ and the group homomorphism $\chi_{\varphi}$ which makes the following diagram commutative;
	\begin{center}
		\begin{tikzcd}
			& C^*(G) \arrow[r,"\varphi"] \arrow[d,swap,"q"] & \C \\
			& C^*(G_{x_{\varphi}}) \arrow[ur,"\chi_{\varphi}",swap],
		\end{tikzcd}
	\end{center}
	where $q\colon C^*(G)\to C^*(G_{x_{\varphi}})$ is the *-homomorphism obtained in Proposition \ref{canonical surjection}.
	
	By the universality of $G_{x_{\varphi}}^{\ab}\defeq (G_{x_{\varphi}})^{\ab} =(G^{\ab})_{x_{\varphi}}$,
	we obtain the group homomorphism $\bar{\chi}_{\varphi}\colon G_{x_{\varphi}}^{\ab}\to \T$ which makes the following diagram commutative;
	\begin{center}
		\begin{tikzcd}
			& C^*(G_{x_{\varphi}}) \arrow[r,"\chi_{\varphi}"] \arrow[d,swap,"q'"] &\C \\
			& C^*(G_{x_{\varphi}}^{\ab}) \arrow[ur,"\bar{\chi}_{\varphi}",swap],
		\end{tikzcd}
	\end{center}
	where $q'\colon C^*(G_{x_{\varphi}})\to C^*(G_{x_{\varphi}}^{\ab})$ denotes the *-homomorphism induced by the quotient map $G_{x_{\varphi}}\to G_{x_{\varphi}}^{\ab}$.
	
	Let $\mathrm{res}\colon C^*(G^{\ab})\to C^*(G^{\ab}_{x_{\varphi}})$ denote the *-homomorphism obtained by the restriction $\mathcal{C}(G^{\ab})\to \mathcal{C}(G^{\ab}_{x_{\varphi}})$ (see Proposition \ref{canonical surjection}).
	Now, we have the following commutative diagram;
	\begin{center}
		\begin{tikzpicture}[auto]
		\node (CG) at (0, 2) {$C^*(G)$}; \node (CGx) at (3, 2) {$C^*(G_{x_{\varphi}})$};
		\node (CGab) at (0, 0) {$C^*(G^{\ab})$};   \node (CGabx) at (3, 0) {$C^*(G^{\ab}_{x_{\varphi}})$.};
		\node (C) at (6, 2) {$\C$};
		\draw[->] (CG) to node {$\scriptstyle q$} (CGx);
		\draw[->, bend left=30pt] (CG) to node {$\scriptstyle \varphi$} (C);
		\draw[->] (CGx) to node {$\scriptstyle \chi_{\varphi}$} (C);
		\draw[->] (CG) to node[swap] {$\scriptstyle \pi$} (CGab);
		\draw[->](CGab) to node {$\scriptstyle \mathrm{res}$} (CGabx);
		\draw[->](CGabx) to node[swap] {$\scriptstyle \bar{\chi}_{\varphi}$} (C);
		\draw[->] (CGx) to node {$\scriptstyle q'$} (CGabx);
		\end{tikzpicture}
	\end{center}
	In particular,
	we have $\varphi=(\bar{\chi}_{\varphi}\circ\mathrm{res})\circ \pi$ and $\bar{\chi}_{\varphi}\circ\mathrm{res}\in \chara(C^*(G^{\ab}))$.
	Hence,
	$\Phi$ is surjective.
	\qed
\end{proof}

Now, we are ready to calculate the abelianization of $C^*(G)$.
\begin{thm}\label{abelianization of C*G}
	Let $G$ be an \'etale groupoid.
	Then,
	$C^*(G)^{\ab}$ is isomorphic to $C^*(G^{\ab})$ via the unique isomorphism $\bar{\pi}$ which makes the following diagram commutative;
	\begin{center}
		\begin{tikzcd}
			& C^*(G) \arrow[r,"\pi"] \arrow[d,"Q",swap] &C^*(G^{\ab}) \\
			& C^*(G)^{\ab} \arrow[ur,"\bar{\pi}",swap],
		\end{tikzcd}
	\end{center}
	where $Q\colon C^*(G)\to C^*(G)^{\ab}$ denotes the quotient map.
\end{thm}

\begin{proof}
	By the universality of $C^*(G)^{\ab}$,
	we obtain a *-homomorphism which makes the following diagram commutative;
	\begin{center}
		\begin{tikzcd}
			& C^*(G) \arrow[r,"\pi"] \arrow[d,"Q",swap] &C^*(G^{\ab}) \\
			& C^*(G)^{\ab} \arrow[ur,"\bar{\pi}",swap].
		\end{tikzcd}
	\end{center}
	It is clear that $\bar{\pi}$ is surjective.
	We show that $\bar{\pi}$ is injective.
	Suppose that $a\in C^*(G)$ satisfies $\pi(a)=0$.
	It suffices to show $Q(a)=0$,
	which is equivalent to $\bar{\varphi}(Q(a))=0$ for all $\bar{\varphi}\in\chara(C^*(G)^{\ab})$ since $C^*(G)^{\ab}$ is commutative.
	Take $\bar{\varphi}\in\chara(C^*(G)^{\ab})$ and define $\varphi\defeq \bar{\varphi}\circ Q$.
	Then, by Lemma \ref{character of CGab}, there exists $\tilde{\varphi}\in \chara(C^*(G^{\ab}))$ which makes the following diagram commutative;
	\begin{center}
		\begin{tikzcd}
			& C^*(G) \arrow[r,"\pi"]  \arrow[dr,"\varphi"] & C^*(G^{\ab}) \arrow[d,"\tilde{\varphi}"] \\
			&  &\C.
		\end{tikzcd}
	\end{center}
	Now,
	we have the following commutative diagram;
	\begin{center}
		\begin{tikzcd}
			& C^*(G) \arrow[r,"\pi"] \arrow[d,"Q",swap] \arrow[dr,"\varphi"] & C^*(G^{\ab})\arrow[d,"\tilde{\varphi}"] \\
			& C^*(G)^{\ab}\arrow[r,"\bar{\varphi}"] &\C .
		\end{tikzcd}
	\end{center}
	Hence, we have $\bar{\varphi}(Q(a))=\tilde{\varphi}(\pi(a))=0$.
	\qed
\end{proof}

\subsection{Duals of \'etale abelian group bundles}
Let $G$ be an \'etale groupoid.
Since the abelianization of $C^*(G)$ is a commutative C*-algebra,
$C^*(G)^{\ab}$ is isomorphic to $C_0(\chara(C^*(G)^{\ab}))$ via the Gelfand transformation.
In this section,
we calculate the Gelfand spectrum $\chara(C^*(G)^{\ab})$.

For a discrete abelian group $\Gamma$,
its Pontryagin dual group is defined as the set of all group homomorphisms from $\Gamma$ to $\T$,
which is denoted by $\widehat{\Gamma}$.
Then, $\widehat{\Gamma}$ is an abelian group with respect to the pointwise multiplication.
It is known that $\widehat{\Gamma}$ is a compact abelian topological group with respect to the topology of pointwise convergence.

\begin{prop}
	Let $\Gamma$ be a discrete group and $Q\colon C^*(\Gamma)\to C^*(\Gamma^{\ab})$ be the *-homomorphism induced by the quotient map $\Gamma\to\Gamma^{\ab}$.
	Then,
	a map 
	\[
	\Phi\colon \widehat{\Gamma^{\ab}} \ni \chi \mapsto \chi\circ Q\in \chara (C^*(\Gamma))  
	\]
	is a homeomorphism.
	Hence, $C^*(\Gamma)^{\ab}$ is isomorphic to $C(\widehat{\Gamma^{\ab}})$.
\end{prop}
\begin{proof}
	This follows from the universality of $\Gamma^{\ab}$ and $C^*(\Gamma)$.\qed
\end{proof}

As seen in the previous proposition,
the key to calculate $\chara(C^*(G))$ is the Pontryagin dual.

\begin{defi}\label{definition of dual groupoid}
	Let $G$ be an \'etale abelian group bundle.
	We define a group bundle $\widehat{G}\defeq \{(\chi, x)\mid x\in G^{(0)}, \chi\in\widehat{G_x}\}$ over $G^{(0)}$.
\end{defi}
Note that $\widehat{G}$ is a group bundle such that $\widehat{G}_x=\widehat{G_x}\times\{x\} (\simeq \widehat{G_x})$ for every $x\in G^{(0)}$.

Let $G$ be an \'etale abelian group bundle and $(\chi,x)\in \widehat{G}$.
Recall that we obtain the *-homomorphism $\varphi_{x,\chi}\in \chara(C^*(G))$ as in Definition \ref{def of phix,chi}.

\begin{defi}
	Let $G$ be an \'etale abelian group bundle.
	For each $f\in \mathcal{C}(G)$,
	we define $\ev_f\colon \widehat{G}\to \C$ by $\ev_f((\chi,x))=\varphi_{x,\chi}(f)$, where $(\chi,x)\in \widehat{G}$.
	We define a topology of $\widehat{G}$ as the weakest topology in which $\ev_f$ is continuous for all $f\in \mathcal{C}(G)$.
\end{defi}

\begin{prop}\label{dual group bundle}
	Let $G$ be an \'etale abelian group bundle.
	Then, the map\footnote{See Proposition \ref{def of fixed pt} and \ref{def of chi} for the definition of $x_{\varphi}$ and $\chi_{\varphi}$. }
	\[\Psi\colon\chara(C^*(G))\ni \varphi \mapsto (\chi_{\varphi},x_{\varphi})\in \widehat{G}\]
	is a homeomorphism.
	Hence,
	$C^*(G)$ is isomorphic to $C_0(\widehat{G})$
	
\end{prop}

\begin{proof}
	Proposition \ref{charaC*G} states that $\Psi$ is a bijection and $\Psi^{-1}$ is given by $\Psi^{-1}((\chi,x))=\varphi_{x,\chi}$ for each $(\chi,x)\in \widehat{G}$.
	For each $f\in \mathcal{C}(G)$,
	a map $\chara(C^*(G))\ni\varphi\mapsto \ev_f((\chi_{\varphi},x_{\varphi}))=\varphi(f)\in \C$ is continuous.
	This means that $\Psi$ is continuous.
	The continuity of $\Psi^{-1}$ follows from approximation arguments.	
	Therefore,
	$\Psi$ is a homeomorphism.
	\qed
\end{proof}
Let $G$ be an \'etale groupoid.
Recall that $G^{\ab}$ is an \'etale abelian group bundle.

\begin{cor}\label{Main theorem 2}
	Let $G$ be an \'etale groupoid.
	Then,
	$C^*(G)^{\ab}$ is isomorphic to $C_0(\widehat{G^{\ab}})$.
\end{cor}

\begin{proof}
	Recall that $C^*(G)^{\ab}$ is isomorphic to $C^*(G^{\ab})$ by Theorem \ref{abelianization of C*G}.
	Since $G^{\ab}$ is an \'etale abelian group bundle,
	Proposition \ref{dual group bundle} implies that $C^*(G^{\ab})$ is isomorphic to $C_0(\widehat{G^{\ab}})$.
	\qed
\end{proof}

\begin{prop}
	Let $G$ be an \'etale abelian group bundle.
	Then,
	$\widehat{G}$ is a locally compact Hausdorff topological group bundle.
	Furthermore,
	$\widehat{G}$ is compact if and only if $G^{(0)}$ is compact. 
\end{prop}

\begin{proof}
	It is clear that $\widehat{G}$ is locally compact Hausdorff,
	since $\widehat{G}$ is homeomorphic to $\chara(C^*(G))$.
	In order to show the continuity of the operations,
	take $f\in \mathcal{C}(G)$ arbitrarily.
	Then,
	the map $\widehat{G}^{(2)}\ni (\chi_1,\chi_2)\mapsto \ev_f(\chi_1\chi_2)=\ev_f(\chi_1)\ev_f(\chi_2)\in \C$ is continuous.
	Therefore,
	the multiplication of $\widehat{G}^{(2)}$ is continuous.
	Similarly,
	one can show that the inverse is continuous.
	Hence,
	$\widehat{G}$ is a locally compact Hausdorff topological group bundle.
	The last assertion follows from the fact that $G^{(0)}$ is compact if and only if $C^*(G)\simeq C_0(\widehat{G})$ is unital.
	\qed 
\end{proof}

\begin{ex}
	Let $\mathfrak{S}_3=\langle s,t\mid s^3=t^2=e, st=ts^2 \rangle=\{e,s,s^2,t,ts,ts^2\}$ be the symmetric group of degree $3$
	and $A_3\defeq\{e,s,s^2\}\subset \mathfrak{S}_3$ be the subgroup of even permutations.
	Let $G\defeq\mathfrak{S}_3\times[0,1]\setminus\{(t,1)\mid t\not\in A_3\}$ be an \'etale group bundle over $[0,1]$. 
	Then, $G$ can be drawn as follows;
	
	
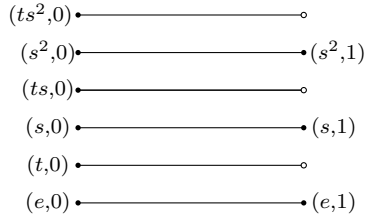
\begin{center}
		\begin{tikzpicture}[auto]
		\node (0s2) at (0,1.5) {}; \node (1s2) at (3,1.5) {};
		\node (0s) at (0,0.5) {}; \node (1s) at (3,0.5) {};
		\node (0) at (0, 0) {}; \node (1) at (3, 0) {};
		\draw[-] (0,0) to node {} (3,0);
		\draw[-] (0,0.5) to node {} (2.97,0.5);
		\draw[-] (0,1) to node {} (3,1);
		\draw[-] (0,1.5) to node {} (2.97,1.5);
		\draw[-] (0,2) to node {} (3,2);
		\draw[-] (0,2.5) to node {} (2.97,2.5);
		\draw[-] (0,1.5) to node {} (2.97,1.5);
		\fill[black] (0) circle (1pt);
		\fill[black] (0s2) circle (1pt);
		\draw[black] (1s2) circle (1pt);
		\draw[black] (3,0.5) circle (1pt);
		\draw[black] (3,2.5) circle (1pt);
		\fill[black] (1) circle (1pt);
		\fill[black] (3,1) circle (1pt);
		\fill[black] (3,2) circle (1pt);
		\fill[black] (0,0.5) circle (1pt);
		\fill[black] (0,1) circle (1pt);
		\fill[black] (0,2) circle (1pt);
		\fill[black] (0,2.5) circle (1pt);
		\node at (3.4,1) {$\scriptstyle (s,1)$};
		\node at (3.45,2) {$\scriptstyle (s^2,1)$};
		\node at (3.4,0) {$\scriptstyle (e,1)$};
		\node at (-0.4,0) {$\scriptstyle (e,0)$}; 
		\node at (-0.4,1.5) {$\scriptstyle (ts,0)$};
		\node at (-0.4,1) {$\scriptstyle (s,0)$};
		\node at (-0.4,2) {$\scriptstyle (s^2, 0)$};
		\node at (-0.4,0.5) {$\scriptstyle (t,0)$};
		\node at (-0.5,2.5) {$\scriptstyle (ts^2,0)$};
		\end{tikzpicture}	
\end{center}
One can see that $[G,G]\subset G$ is not closed.
By Proposition \ref{Hausdorffness of G/H},
$G^{\ab}=G/[G,G]$ is not Hausdorff.
Indeed, letting $q\colon G\to G^{\ab}$ denote the quotient map, $G^{\ab}$ looks like the following;
	
	
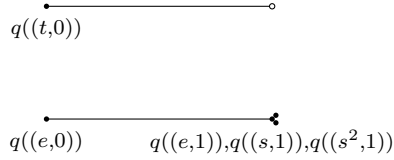
\begin{center}
		\begin{tikzpicture}[auto]
		\node (0s) at (0,1.5) {}; \node (1s) at (3,1.5) {};
		\node (0) at (0, 0) {}; \node (1) at (3, 0) {};
		\draw[-] (0,0) to node {} (3,0);
		\draw[-] (0,1.5) to node {} (2.97,1.5);
		\fill[black] (0) circle (1pt);
		\fill[black] (0s) circle (1pt);
		\draw[black] (1s) circle (1pt);
		\fill[black] (1) circle (1pt); \fill[black] (3.05, 0.05) circle (1pt);
		\fill[black] (3.05,-0.05) circle (1pt);
		\node at (3,-0.3) {$\scriptstyle q((e,1)),q((s,1)),q((s^2,1))$};
		\node at (0,-0.3) {$\scriptstyle q((e,0))$}; 
		\node at (0,1.2) {$\scriptstyle q((t,0))$};
		\end{tikzpicture}		
	\end{center}
	The dual $\widehat{G^{\ab}}$ of $G^{\ab}$ can be drawn as follows;
	\begin{center}
		\begin{tikzpicture}[auto]
		\node (0s) at (0,1.5) {}; \node (1s) at (3,1.5) {};
		\node (0) at (0, 0) {}; \node (1) at (3, 0) {};
		\draw[-] (0,0) to node {} (3,0);
		\draw[-] (0,1.5) to node {} (3,0);
		\fill[black] (0) circle (1pt);
		\fill[black] (0s) circle (1pt);
		\draw[black] (1s) circle (1pt);
		\fill[black] (1) circle (1pt); \fill[black] (3, 0.7) circle (1pt);
		\fill[black] (3,1.5) circle (1pt);
		\end{tikzpicture}
	\end{center}
	Note that $\widehat{G^{\ab}}$ is not \'etale.
\end{ex}

\bibliographystyle{plain}
\bibliography{bunken20180828}

\end{document}